\documentclass[12pt,reqno]{amsart}

\usepackage{amsrefs}
\usepackage[colorlinks]{hyperref}
\hypersetup{citecolor = blue}
\usepackage{setspace}
\theoremstyle{plain}
\newtheorem{thm}{Theorem}[section]
\newtheorem{prp}[thm]{Proposition}
\newtheorem{cor}[thm]{Corollary}
\newtheorem{lem}[thm]{Lemma}

\newtheorem{defin}[thm]{Definition}

\newcommand{\C}{\mathbb{C}}
\newcommand{\R}{\mathbb{R}}
\newcommand{\N}{\mathbb{N}}
\newcommand{\Q}{\mathbb{Q}}

\newcommand{\dist}{\,\mathrm{dist}}

\newcommand{\U}{\mathcal{U}}

\def\be{\begin{equation}}
\def\ee{\end{equation}}

\usepackage{enumitem}

\def\RR{\mathbb{R}}
\def\CC{\mathbb{C}}
\newcommand{\sq}{\mathrel{\square}}
\newcommand{\Ran}{\ensuremath{\mathbb{R}_{\rm an}}}
\newcommand{\Ranp}{\ensuremath{\mathbb{R}_{\rm an}^{\rm pow}}}
\def\cL{{\mathcal L}}
\newcommand{\cF}{\ensuremath{\mathcal{F}}}
\def\cX{{\mathcal X}}
\def\cU{{\mathcal U}}
\def\cS{{\mathcal S}}
\def\cG{{\mathcal G}}

\begin{document}

\title[Doubling coverings]{Doubling coverings via resolution of singularities and preparation}

\author{Raf Cluckers}
\address{Universit\'e de Lille, Laboratoire Painlev\'e, CNRS - UMR 8524, Cit\'e Scientifique, 59655
Villeneuve d'Ascq Cedex, France, and,
KU Leuven, Department of Mathematics,
Celestijnenlaan 200B, B-3001 Leu\-ven, Bel\-gium}
\email{raf.cluckers@univ-lille.fr}
\urladdr{http://rcluckers.perso.math.cnrs.fr/}

\author{Omer Friedland}
\address{Institut de Math\'ematiques de Jussieu - Paris Rive Gauche, Sorbonne Universit\'e - Campus Pierre et Marie Curie, 4 place Jussieu, 75252 Paris, France.}
\email{omer.friedland@imj-prg.fr}

\author{Yosef Yomdin}
\address{Department of Mathematics, The Weizmann Institute of Science, Rehovot 76100, Israel.}
\email{yosef.yomdin@weizmann.ac.il}

\thanks{R.C. was partially supported by the European Research Council under the European Community's Seventh Framework Programme (FP7/2007-2013) with ERC Grant Agreement nr. 615722 MOTMELSUM, by the Labex CEMPI (ANR-11-LABX-0007-01), and by KU Leuven IF C14/17/083. \ Y.Y. was partially supported by the Minerva foundation.}

\begin{abstract}
In this paper we provide asymptotic upper bounds on the complexity in two (closely related) situations. We confirm {\it for the total doubling coverings and not only for the chains} the expected bounds of the form
$$
\kappa(\U) \le K_1(\log ({1}/{\delta}))^{K_2} .
$$
This is done in a rather general setting, i.e. for the $\delta$-complement of a polynomial zero-level hypersurface $Y_0$ and for the regular level hypersurfaces $Y_c$ themselves with no assumptions {\it on the singularities} of $P$. The coefficient $K_2$ is the ambient dimension $n$ in the first case and $n-1$ in the second case. However, the question of a uniform behavior of the coefficient $K_1$ remains open.
As a second theme, we confirm {\it in arbitrary dimension the upper bound for the number of a-charts} covering a real semi-algebraic set $X$ of dimension $m$ away from the $\delta$-neighborhood of a lower dimensional set $S$, with bound of the form
$$
\kappa(\delta) \le C (\log ({1}/{\delta}))^{m}
$$
holding uniformly in the complexity of $X$. We also show an analogue for level sets with parameter away from the $\delta$-neighborhood of a low dimensional set.  More generally, the bounds are obtained also for real subanalytic and real power-subanalytic sets.  
\end{abstract}
 
\maketitle

\section{Introduction} \label{sec:intro}

Let us recall the definition of a doubling covering, as given in \cite{FY17,FY}. Let $Y$ be a complex $n$-dimensional manifold and let $G\subset Y$ be a relatively compact domain in $Y$. Let $B_1^n$ be the unit ball in $\C^n$. For $\gamma >1$, a $\gamma$-doubling covering $\U$ of $G$ in $Y$ is a finite collection of analytic univalent functions $\psi_j:B_1^n \to Y$ satisfying the following conditions:

\noindent 1. The images (aka charts) $U_j = \psi_j(B_1^n)$ cover the closure $\bar G$ of $G$.

\noindent 2. Each $\psi_j$ is extendible to a mapping $\tilde\psi_j:B_\gamma^n \to Y$, which is univalent in a neighborhood of $B_\gamma^n$, where $B_\gamma^n \subset\C^n$ is the $\gamma$-times larger concentric ball of $B_1^n$.


For $\gamma = 2$ we may omit $\gamma$ in notations, and call a covering just a doubling one (sometimes using this short name also for $\gamma$-doubling coverings with $\gamma \ne 2$). Recall also that a doubling chain $Ch$ joining two points $v_1,v_2\in Y$ is a series of doubling charts $\psi_j$, $j = 1,\dots,l$, so that their images $U_j = \psi_j(B_1^n)$ satisfy $U_j\cap U_{j+1}\ne \emptyset$, $j = 1,\dots, l-1$, and $v_1\in U_1, v_2\in U_l$. We denote by $l(Ch)$ the length of a chain $Ch$, that is, the number of its elements.

Doubling coverings provide a conformally invariant version of Whitney's ball coverings of a domain $W\subset \R^n$, introduced in \cite{Whi}. These coverings consist of balls $B_j$ so that larger concentric balls $\gamma B_j$ are still in $W$. In our definition we replace $W$ by a complex manifold $Y$, while the balls $B_j$ are replaced by the charts $U_j$. In \cite{FY17} we prove, in a rather general form, that the doubling coverings (more accurately, the chains of doubling charts) on $Y$ provide an upper bound to the Kobayashi metric and an upper bound to the ``doubling constants'' on this manifold (see also \cite{FY}). Thus, these facts suggest possible connections with complex hyperbolic geometry. The results on quasi-hyperbolic metrics, on one side and on the complexity of Whitney's ball coverings, on the other, obtained in \cite{MV87} and in other related publications, look very relevant (see also the survey \cite[Chapter 6]{BB12} for extensions and developments of Whitney's coverings in other directions).

In view of these connections, one can hope that doubling coverings on $Y$ provide a common ground for a better understanding of the above mentioned structures. Consequently, one of the most important problems related to doubling coverings $\U$ of $G$ in $Y$ is an explicit bound on their ``complexity'', $\kappa(\U)$, which is the number of the doubling charts in $\U$.

Let us stress that the mere existence of a finite doubling covering $\U$ for any regular complex manifold $Y$ and any compact $G\subset Y$ is immediate: we just use the coordinate charts on $Y$. Moreover, for singular $Y$ (a situation not addressed in this paper) this fact remains basically true. Indeed, in situations where the resolution of singularities works (algebraic, analytic, sub-analytic and some o-minimal settings), we just double-cover a ``non-singular model'' of $G$ and $Y$, and compose the charts with the resolution mapping $\sigma$. However, the complexity may blow up in families. ``Uncontrolled'' complexity growth may present a major problem in applications, while the power-logarithmic bounds obtained below promise to work.

Now, we can explain the nature of the difficulties we settle in this paper. Let us start with Whitney's ball coverings. In this case there are rather accurate bounds on the complexity of such coverings. In particular, in \cite{MV87} some bounds on the complexity of the ball coverings of the complements of closed sets are given, in terms of the Minkowski dimension of these sets. For a set $A$ of dimension $l$ a (rather accurate) bound on Whitney's ball covering $\U$ of $B_1^n \setminus A_\delta$ is of the form
\be \label{eq.Vuorinen}
\kappa(\U) \le K(\frac{1}{\delta})^l\log ({1}/{\delta}).
\ee

Compare also with \cite{FY17} where a similar bound for Whitney's ball covering of the punctured disk was given (uniform in the geometry of the deleted points). Easy examples (for instance, $A$ being a hyperplane) show that these bounds are sharp and cannot be improved {\it for Whitney's balls}. The factor $({1}/{\delta})^l$ in the bound of \eqref{eq.Vuorinen} is too big for the intended applications. However, {\it if we replace the doubling Whitney's balls with their holomorphic images} (as we do in the definition of doubling coverings), we can hope to get a bound of the form
\be \label{eq.our.conjecture}
\kappa(\U) \le K_1(\log ({1}/{\delta}))^{K_2}.
\ee

This kind of bounds were conjectured (in different forms) in \cite{FY17,FY,Y87,Y91,Y08,Y15}. A special case was settled in \cite{FY17}, where we confirm the expected bound (with $K_2 = 1$!) in case of regular level hypersurfaces $Y_c = \{P = c\}$ for polynomials $P$ with non-degenerated critical points. However, the method of \cite{FY17} cannot be directly extended to the polynomials $P$ with non-isolated singularities.

On the other hand, in \cite{FY} we prove the bound of the form \eqref{eq.our.conjecture}, (also with $K_2 = 1$!), for the length of the ``doubling chains'', joining any two points in the $\delta$-complement of a zero-level hypersurface of a polynomial. Thus, for the length of the chains and hence, for the Kobayashi distance, the bound \eqref{eq.our.conjecture} was, essentially, confirmed with $K_2 = 1$ and with $K_1$ depending only on the {\it degree $d$} of $P$.

Let us introduce some notatios. Let $P(z) = \sum_{|\alpha| \le d}a_\alpha z^\alpha$ be a complex polynomial of degree $d$ in $\C^n$ written in the usual multi-index notations $z = (z_1,\ldots,z_n) \in \C^n$, $\alpha = (\alpha_1,\ldots,\alpha_n)\in \N^n$, $|\alpha| = \sum_{i = 1}^n |\alpha_i|$ and $z^\alpha = z_1^{\alpha_1}\cdots z_n^{\alpha_n}$. We say that $P$ is normalized if $\|P\| : = \sum_{|\alpha| \le d}|a_\alpha| = 1$. We denote by
$$
Y_c = \{P = c\}\subset \C^n
$$
the $c$-level hypersurface of $P$.

In this paper we provide asymptotic upper bounds on the complexity in two (closely related) situations, we confirm {\it for the total doubling coverings and not only for the chains} the expected bounds of the form \eqref{eq.our.conjecture}. This is done in a rather general setting, i.e. for the $\delta$-complement of a polynomial zero-level hypersurface $Y_0$ and for the regular level hypersurfaces $Y_c$ themselves with no assumptions {\it on the singularities} of $P$. The coefficient $K_2$ in \eqref{eq.our.conjecture} is the ambient dimension $n$ in the first case and $n-1$ in the second case. However, the question of a uniform behavior of the coefficient $K_1$ in (\ref {eq.our.conjecture}) remains open.

\begin{thm} [Complement of zero-level hypersurfaces] \label{thm-complement}
Let $Y_0$ be the zero-level hypersurface of $P$ and let $G_\delta = B_1^n\setminus Y_0^\delta$, where $Y_0^\delta$ is a $\delta$-neighborhood of $Y_0$ for $0<\delta<\kappa_1$. There exists a doubling covering $\U$ of $G_\delta$ in $\C^n\setminus Y_0$ so that
$$
\kappa(\U) \le C_1 (\log({c_1}/{\delta}))^n ,
$$
where $C_1, c_1, \kappa_1>0$ are constants depending on $n,d$ and $P$.
\end{thm}

\begin{thm} [Regular level hypersurfaces] \label{thm-regular}
Let $Y_c$ be a regular level hypersurface of $P$ and let $\bar Y_c = Y_c \cap B_1^n$. Let $\rho$ be the distance of $c$ to the set of singular values of $P$. We assume that $0<\rho<\kappa_2$. There exists a doubling covering $\U$ of $\bar Y_c$ in $Y$ so that
$$
\kappa(\U) \le C_2 (\log({c_2}/{\rho}))^{n-1} ,
$$
where $C_2, c_2, \kappa_2>0$ are constants depending on $n,d$ and $P$.
\end{thm}

In Section \ref{sec-thms-1.1-and-1.2} we provide a doubling covering for $G_\delta$ in the monomial case, and we do so also for $\bar Y_c$. On this base, using resolution of singularities, we prove Theorems \ref{thm-complement} and \ref{thm-regular}.

In a somewhat different setting and under the name ``analytic parametrizations'', doubling coverings were essentially introduced in \cite{Y91} as a tool for handling topological entropy and other similar dynamical invariants, of real analytic mappings. We refer to \cite{FY17,FY,Y91,Y08,Y15} for further developments and for some discussions on the connections with bounding the density of rational points on analytic varieties in diophantine geometry and other applications.

In Section \ref{sec:real} we study this slightly different setting of ``analytic parametrizations'' using a-charts, recalled in Definition \ref{defn:acu} and introduced first in \cite[Definition 2.1]{Y91} under the name acu's. We provide analogues of Theorems \ref{thm-complement} and \ref{thm-regular} for real semi-algebraic sets and with a-charts.

Let us give our main results. Write $I$ for the real interval $[-1,1]$ and for each $r>0$ write $\Delta_r$ for the complex disk $\{z\in\CC\mid |z|\leq r\}$.
As usual, for a subset $A\subset \RR^\ell$, we call a function $f:A\to\RR^n$ real analytic if there exists an open neighborhood $O$ of $A$ and a real analytic function $O\to\RR^n$ whose restriction to $A$ is $f$.
We recall Definition 3.1 from \cite{Y08}, where they are called analytic $1$-chart in full and a-charts in short.
\begin{defin}[a-charts] \label{defn:acu}
A real analytic mapping $\psi: I^\ell\to \RR^n$ is called an a-chart if it can be extended to a holomorphic mapping $\tilde \psi:\Delta_3^\ell\to \CC^n$ such that moreover
$\tilde \psi(z)- \psi(0)$ lies in $\Delta_1^n$ for each $z\in \Delta_3^\ell$.
\end{defin}

For a set $S\subset \RR^m$ and $\delta>0$, by the $\delta$-neighborhood of $S$ we mean the set of points $x\in\RR^m$ that lie at distance at most $\delta$ to $S$, and we write $S_\delta$ to denote this tube. Here, the distance between $x$ and $S$ is defined as the infimum over all $s\in S$ of $\max_{i=1}^m |x_i-s_i|$.

The following is a variant in general dimension of Theorem 3.1 of \cite{Y08} and of the complex case of Theorem \ref{thm-complement} above. Note that it is more uniform than Theorem \ref{thm-complement}.

\begin{thm} \label{thm:semi-alg-a-chart}
Let $X\subset I^n$ be a semi-algebraic set of dimension $m>0$. Then, there exist a semi-algebraic subset $S\subset I^n$ of dimension $<m$ and a constant $C$ such that the following holds. For each $\delta>0$ with $\delta\leq 1$ there are semi-algebraic a-charts
$$
\psi_i:I^m\to \RR^n \mbox{ for } i=1,\ldots, \kappa(\delta)
$$ 
with
$$
 \kappa(\delta)\leq C(\log 1/\delta)^m
$$
such that the union of the $\psi_i(I^m)$ contains $X\setminus S_\delta$, where $S_\delta$ is the $\delta$-neighborhood of $S$.
Furthermore, $C$ and the complexity of $S$ are bounded in terms of the complexity of $X$.
\end{thm}

Next comes our analogue for real semi-algebraic sets of any dimension of the complex result of Theorem \ref{thm-regular} above. Again, it is more uniform than Theorem \ref{thm-regular}, and, more flexible in the dimension of the family parameters.

\begin{thm} \label{thm:fibers-semi-alg-a-chart}
Let $f:I^n\to I^k$ be a semi-algebraic function. Suppose that the nonempty fibers of $f$ have dimension $m$. Then, there exist a constant $C$ and a subanalytic set $S\subset I^k$ of dimension less than $k$ such that, for any $\delta>0$ with $\delta\leq 1$ and for any $c\in I^k$ of distance at least $\delta$ to $S$ there are a-charts
$$
\psi_i:I^{m}\to \RR^n \mbox{ for } i=1,\ldots, \kappa(\delta)
$$
with
$$
\kappa(\delta)\leq C \log (1/\delta)^{m}
$$
such that the union of the $\psi_i(I^{m})$ contains $f^{-1}(c)$.
\end{thm}

The proofs of Theorems \ref{thm:semi-alg-a-chart} and \ref{thm:fibers-semi-alg-a-chart} and the definition of subanalytic sets are given in Section \ref{sec:real}, as well as their corresponding generalizations for subanalytic and power-subanalytic sets as Theorems \ref{thm:semi-a-chart:g} and \ref{thm:fibers-a-chart:g:g}.
It may be interesting to see whether $S$ can be taken semi-algebraic as well in Theorem \ref{thm:fibers-semi-alg-a-chart} (see also the two questions at the very end of the paper).

\section{Proof of theorems \ref{thm-complement} and \ref{thm-regular}} \label{sec-thms-1.1-and-1.2}

The main idea behind the proofs of these two results, is to make a reduction from the general case of a polynomial $P$ of degree $d$ in $\C^n$ to the monomial case. This is done by applying the following basic version of resolution of singularities (see, e.g. \cite{BM91}).

\begin{thm} \label{Res.Sing}
Let $P$ be a polynomial of degree $d$ in $\C^n$, and let $Y_0$ be its zero-level hypersurface. There exist a regular $n$-dimensional algebraic variety $X$ and a proper mapping $\sigma:X\to \C^n$ so that for any point $y\in \sigma^{-1}(Y_0) \subset X$ there is a neighborhood $W_y$ of $y$ in $X$ and a local coordinate system $x_1,\ldots,x_n$ in $W_y$, in which
$$
P\circ \sigma (x) = U(x) x^\alpha , \quad \forall x\in W_y
$$
where $x^\alpha = \prod_{i = 1}^n x_i^{\alpha_i}$, $\alpha = (\alpha_1,\ldots,\alpha_n)\in\N^n$, \ $\min_{i}\alpha_i \ge 1,$ and $U(x)$ is a non-vanishing function (clearly, $U$ depends on $y$). In particular, the preimage $\sigma^{-1}(Y_0)$ coincides locally with the union $Z$ of the coordinate hyperplanes $Z_i = \{x_i = 0\}$ for $i = 1,\ldots,l_y,$ where $l_y$ is the number of the local coordinates, actually apearing in the monomial $x^\alpha$.
\end{thm}

Accordingly, in order to construct a doubling covering either for $G_\delta = B_1^n\setminus Y_0^\delta$, or for $\bar Y_c = B_1^n \cap Y_c$ (for sufficiently small $\delta$ and $c$), it is enough to construct such coverings in each of a finite number of the neighborhoods $W_y$ in $X$, covering the compact preimage of $Y_0\cap B^n_1.$

In case of $G_\delta$ we have also to cover the part of $G_\delta$ out of the union of $W_y$, but this is immediate, with the number of charts not depending on $\delta$.

\subsection{Proof of Theorem \ref{thm-complement}} \label{sec:compl}

The reduction achieved above allows us to restrict considerations to the following case: in the appropriate system of local coordinates for any $x=(x_1,\ldots,x_n) \in W_y$ we have, as above, $P\circ \sigma (x) = U(x) x^\alpha$. Without lost of generality we may assume that in local coordinates $x_1,\ldots,x_n$ the neighborhood $W_y$ is defined by $|x_i|\le 1$, $i=1,\ldots,n$, and that there is a constant $C_y>0$ so that
\be \label{eq:U}
\frac1{C_y} \le |U(x)| \le C_y , \quad \forall x \in W_y .
\ee
We also assume below that $l_y=n$. The case $l_y<n$ is treated exactly in the same way, with better bounds.

\medskip

The following proposition from \cite{FY} is, essentially, a version of \L ojasiewicz inequality. For our applications it is important to keep all the constant explicit and depending only on $n,d$. Notice however, that it is valid only in complex domain.

\begin{prp} \label{prp:dist.value}
Let $P$ be a normalized polynomial of degree $d$ on $\C^n$ and let $Y_0$ be its zero-level hypersurface. Then, for any $x\in B_1^n$ we have
$$
c_{n,d} \dist (x, Y_0)^d \le |P(x)| \le C_{n,d} \dist (x, Y_0) ,
$$
where $c_{n,d}, C_{n,d}>0$ are constants depending only on $n,d$.
\end{prp}

Let $\eta>0$ to be chosen later. For $i = 1,\ldots,n$ we define the set $S_i^\eta$, in local coordinates $x_1,\ldots,x_n$, by $|x_i|<\eta$, and denote by $S^\eta$ a $\eta$-neighborhood of the union of the coordinate hyperplanes $Z$, defined as the union $S^\eta = \bigcup_{i = 1}^n S_i^\eta$.

\begin{cor} \label{cor:monomial1}
Let $\delta >0$, $\alpha_0 = \min_{i = 1}^n \alpha_i \ge 1$ and put $\eta = (\frac{c_{n,d}}{C_y}\delta^d)^{\frac{1}{\alpha_0}}$. Then,
$$
S^\eta \subset \sigma^{-1}(Y_0^\delta) \cap W_y,
$$
where, as above, $Y_0^\delta$ is a $\delta$-neighborhood of $Y_0$.
\end{cor}

\begin{proof}
By Proposition \ref{prp:dist.value} for any $x\in B_1^n$ we have
$$
c_{n,d} \dist (x, Y_0)^d \le |P(x)| .
$$
Let $\delta' = c_{n,d}\delta^d$ and denote by $\hat Y_0^{\delta'}$ the sublevel set $\{|P(x)|<\delta'\}$. Then, for any point $x \in \hat Y_0^{\delta'}$ we have $|P(x)|<\delta'$ and therefore $\dist (x, Y_0) \le (\frac{\delta'}{c_{n,d}})^{\frac{1}{d}} = \delta$, i.e. $x\in B_1^n \cap Y_0^\delta$. Thus, $B_1^n \cap \hat Y_0^{\delta'} \subset B_1^n \cap Y_0^\delta$. We conclude that the preimage $\sigma^{-1}(Y_0^\delta) \cap W_y$ contains the subset of $W_y$ defined by the inequality
$$
|P\circ \sigma (x)| = |U(x) x^\alpha|<\delta'.
$$
But for any $x\in S^\eta$, by definition of $S^\eta$ and by \eqref{eq:U}, we have
$$
|U(x)\cdot x_1^{\alpha_1}\cdot x_2^{\alpha_2}\cdot\ldots\cdot x_n^{\alpha_n}| \le C_y \eta^{\alpha_0} = \delta' ,
$$
and hence $S^\eta \subset \sigma^{-1}(Y_0^\delta) \cap W_y$.
\end{proof}

Therefore, it is sufficient to construct a doubling covering of $W_y\setminus S^\eta$ in $W_y \setminus Z$, where $Z$, as above, is the union of the coordinate hyperplanes $Z_i = \{x_i = 0\}$ for $i = 1,\ldots,n$. Indeed, for all the charts $\psi_j$ of the doubling covering of $W_y\setminus S^\eta$ in $W_y \setminus Z$, the charts $\sigma \circ \psi_j$ will form a doubling covering of $B_1^n\setminus Y_0^\delta$ in $B_1^n\setminus Y_0$. Note also that it is enough to consider the case where $y = 0$ and $W_y$ is the unit polydisc
$$
Q_n = D_1\times D_1\times\cdots \times D_1.
$$
We denote the complement of $S^\eta$ in $Q_n$ by $Q_n^\eta : = Q_n\setminus S^\eta$.

Now, we need the following ``model'' result:

\begin{prp} \label{prp:doubl.cov.S.eta}
Let $\eta>0$ and let $\gamma\ge 2$. There exists a $\gamma$-doubling covering $\U$ of $Q_n^\eta$ in $\C^n\setminus Z$ with the following properties:

\noindent 1. Each chart $\psi_j$ of $\U$ is an affine mapping of $B_1^n$ to $\C^n\setminus Z$, extendible, as an affine mapping, to $\tilde \psi_j:B_\gamma^n \to \C^n\setminus Z$.

\noindent 2. The complexity $\kappa(\U)$ does not exceed $(9\gamma^n)^n(\log ({9\gamma^n}/{\eta}))^n$. In particular, for $\gamma=2$,
$$
\kappa(\U)\le(9\cdot 2^n)^n(\log ({9\cdot 2^n}/{\eta}))^n.
$$
\end{prp}

Before proving this proposition, let us first conclude the proof of Theorem \ref{thm-complement}. In order to prove the theorem, we chose a certain finite covering of $\sigma^{-1}(Y_0)$ by the neighborhoods $W_y$, provided by Theorem \ref{Res.Sing}. Then, we apply Proposition \ref{prp:doubl.cov.S.eta} with $\gamma=2$ to each of these neighborhoods $W_y$ separately. However, first we have to normalize $W_y$ to the standard form $Q_n = D_1\times D_1\times\cdots \times D_1$, used in Proposition \ref{prp:doubl.cov.S.eta}. Next, in each $W_y$ we apply Corollary \ref{cor:monomial1}, in order to find the appropriate $\eta$. In these steps the parameter $\delta$ is scaled accordingly. As a result, $\delta$ enters the bound in Theorem \ref{thm-complement} with a coefficient $c_1$, depending on the geometry of the resolution $\sigma$ of Theorem \ref{Res.Sing}, in contrast with Proposition \ref{prp:doubl.cov.S.eta}, where the coefficients are absolute and explicit. The same concerns the coefficient $C_1$, which is obtained by summing the corresponding coefficients over the neighborhoods $W_y$. This completes the proof of Theorem \ref{thm-complement}. $\square$

First, let us sketch the proof of Proposition \ref{prp:doubl.cov.S.eta}. It is done by induction on the dimension. In dimension $n = 1$ the result is a partial case of \cite[Theorem 2.2]{FY17} (see also \cite[Example 2]{FY}), the required covering of $D_1\setminus D_\eta$ in $D_1\setminus \{0\}$ consists of the ``Whitney's disks'', accumulating to the origin. Assume that the required covering $\U$ has been constructed in dimension $n$. We produce the required covering in dimension $n+1$. To achieve this extension we introduce a ``suspension'' construction, extending an $n$-dimensional chart into an $(n+1)$-dimensional one. This name (suspension), in a pretty similar meaning, is traditionally used in algebraic and homotopic topology.

What follows is a definition of a suspension, then we return to the proof of Proposition \ref{prp:doubl.cov.S.eta} below. We assume that $\gamma>1$ is fixed, and have three free parameters: $a \in \C$, $\lambda, \beta \in \R_+$ with $1<\beta < \gamma$. The parameter $\lambda$ defines the ``height'' of the suspension, while $a$ defines its ``vertical'' (in $\C$) shift. The third parameter $\beta$ controls a ``thickening'' of the suspensions, which is necessary to ``suspend'' coverings.

\begin{defin}
Let $Y$ be a complex $n$-dimensional manifold and let $\psi: B_1^n\to Y$ be a $\gamma$-doubling chart. Let $a \in \C$, $\lambda, \beta \in \R_+$ with $1<\beta < \gamma$ be given. The $(\lambda, a, \beta)$-suspension $\Sigma_{\lambda, a, \beta} \psi$ of $\psi$ is an analytic mapping
$$
\Sigma_{\lambda, a,\beta} \psi: B_1^{n+1}\to Y\times \C
$$
defined, for $(x,y)\in B_1^{n+1}$ with $x\in B_1^n$ and $y\in \C$, by
\be \label{eq:suspension}
\Sigma_{\lambda, a, \beta} \psi(x,y) = (\tilde \psi(\beta x), \lambda y+a)\in Y\times \C
\ee
where $\tilde \psi$ is the analytic extension of $\psi$ to $B_\gamma^n$.
\end{defin}

In the following lemma, we summarize some simple properties of the suspension construction.

\begin{lem} \label{prp:doubl.cov.susp}
Let $Y$ be a complex $n$-dimensional manifold and let $\psi: B_1^n\to Y$ be a $\gamma$-doubling chart. Let $a \in \C$, $\lambda, \beta \in \R_+$ with $1<\beta < \gamma$. Then, the $(\lambda, a, \beta)$-suspension $\Sigma_{\lambda, a, \beta} \psi$ of $\psi$ is a $\theta$-doubling chart in $Y\times \C$ with $\theta = {\gamma}/{\beta}$.

Moreover, if $\U$ is a $\gamma$-doubling covering of a compact $G\subset Y$, then the collection of the suspended charts $\Sigma \U = \{\Sigma_{\lambda, a, \beta} \psi : \psi \in \U\}$ forms a $\theta$-doubling covering of $G\times D_\nu^a \subset Y\times \C$, with $D_\nu^a$ a disk of radius $\nu$ centered at $a\in \C$, and $\nu = \lambda \sqrt{1-\frac{1}{\beta^2}}$.
\end{lem}

\begin{proof}
The suspension $\Sigma_{\lambda, a, \beta} \psi$ of $\psi$ is extendible to the concentric ball $B_\theta^{n+1}$ by the same expression \eqref{eq:suspension}. Indeed, since by assumptions, $\beta < \gamma$, for any $x\in B_\theta^n$ we have $\beta x \in B_\gamma^n$ and hence $\tilde \psi(\beta x)$ is well defined and belongs to $Y$. Hence, $\Sigma_{\lambda, a, \beta} \psi$ is a $\theta$-doubling chart in $Y\times \C$.

Now, let $\U$ be a $\gamma$-doubling covering of a compact $G\subset Y$. In order to prove that $\Sigma \U$ forms a $\theta$-doubling covering of $G\times D_\nu^a \subset Y\times \C$, consider a point
$(v,\omega)\in G\times D_\nu^a$. Since $\U$ is a covering of $G$, we have $v=\psi(x)$, for certain $\psi \in \U$ and $x\in B^n_1$. Therefore, by \eqref{eq:suspension} we get
\begin{align*}
\Sigma_{\lambda, a, \beta} \psi(\frac{x}{\beta},y) = (\tilde \psi(\beta x/\beta), \lambda y+a) = (\psi( x), \lambda y+a) = (v, w) .
\end{align*}
We need to check that $(x/\beta,y)\in B_1^{n+1}$. Indeed, $w\in D_\nu^a$ and so $\lambda y\in D_0^a$ and thus $|y|^2\le (\nu/\lambda)^2$. Hence, by our choice of $\nu$, we get
$$
\|x/\beta\|^2+|y|^2 \le 1/\beta^2 + (\nu/\beta)^2 \le 1 ,
$$
which completes the proof of Lemma \ref{prp:doubl.cov.susp}.
\end{proof}

Now, we come to a general statement, concerning the coverings with suspensions. For a given $\mu>1$, starting with a $\mu$-doubling covering $\U$ of a compact $G\subset Y$, we want to cover the compact set $G \times \{D_1 \setminus D_\delta\}$ in $Y \times \{\C \setminus \{0\}\}$.

\begin{prp} \label{prp:doubl.cov.susp2}
Let $1<\beta < \mu$ be given, and put $\theta = \frac{\mu}{\beta}> 1$. Let $Y$ be a complex $n$-dimensional manifold and let $\U$ be a $\mu$-doubling covering of a compact $G\subset Y$. Then, for any $\delta>0$ there exists a $\theta$-doubling covering $\tilde \U$ of $G \times \{D_1 \setminus D_\delta\}$ in $Y \times \{\C \setminus \{0\}\}$ with
$$
\kappa(\tilde \U) \le 3\zeta \log (\frac{3\zeta}{\delta}) \kappa(\U)
$$
where $\zeta =\frac{2\mu}{\beta}\left [1-\frac{1}{\beta^2} \right ]^{-\frac{1}{2}}$.
\end{prp}

\begin{proof}
We construct the required covering $\tilde \U$ as the union of the suspensions $\Sigma_{\lambda_j, a_j, \beta}\U$ of $\U$ over $j = 1,2,\ldots,N$, where $N, \lambda_j, a_j$ are determined below. Thus, $\tilde \U$ is the union of ``layers'', each layer being ``vertically'' (in the direction of the factor $\C$ in $Y\times \C$) shifted and properly rescaled suspension of $\U$. The ``widths'' of $\Sigma_j \U$ decreases exponentially in $j$ and thus we need an order of $\log (\frac{1}{\delta})$ such layers to cover $G \times \{D_1 \setminus D_\delta\}$ (see e.g. a similar construction in \cite[Example 2]{FY}).

More accurately, let us fix $\zeta =2\theta \left [1-\frac{1}{\beta^2} \right ]^{-\frac{1}{2}}=\frac{2\mu}{\beta}\left [1-\frac{1}{\beta^2} \right ]^{-\frac{1}{2}}$, and let
$$
\hat \U =\{\tilde D_j : j=1,\ldots,N\}, \ \text {with} \ N=3\zeta \log (\frac{3\zeta}{\delta}),
$$
be a $\zeta$-covering of $D_1 \setminus D_\delta$ in $\C \setminus \{0\}$, where $\tilde D_j$ are the ``Whitney disks'', provided by Theorem 2.2 of \cite{FY17}.

For $j=1,\ldots,N$ denote by $a_j$ and $r_j$ the center and the radius of the disk $\tilde D_j\in \hat \U$, respectively, and put $\lambda_j=r_j\left [1-\frac{1}{\beta^2} \right ]^{-\frac{1}{2}}$.

We claim that the suspensions $\Sigma_{\lambda_j, a_j, \beta}\psi$ for all $\psi \in \U$, form the required covering. Indeed, by Lemma \ref{prp:doubl.cov.susp}, the collection of the suspended charts $\Sigma_j:=\{\Sigma_{\lambda_j, a_j, \beta} \psi : \psi \in \U\}$ forms a $\theta$-doubling covering of $G\times D_\nu^{a_j} \subset Y\times \C$, with $\nu = \lambda_j \left [1-(\frac{1}{\beta})^2\right ]^{\frac{1}{2}}=r_j$. Thus, $\Sigma_j$ covers $G\times D_{r_j}^{a_j}=D_j$. Since $\tilde D_j$ cover $D_1 \setminus D_\delta$, we conclude that $\tilde \U=\cup_j\Sigma_j$ covers $G \times \{D_1 \setminus D_\delta\}$.

It remains to show that the suspended charts do not touch the zero section $Y\times \{0\}$. Since the disks $D_j$ form a $\zeta$-covering of $D_1 \setminus D_\delta$ in $\C \setminus \{0\}$, for any $j$ we have $r_j < \frac{|a_j|}{\zeta}$ (this is the $\zeta$ doubling condition).

On the other hand, for any $j$ and $\psi \in \U$ consider the projection of the image of the suspension $\Sigma_{\lambda_j, a_j, \beta} \psi$ in $Y \times \C$ to $\C$. By the expression (\ref{eq:suspension}) for $\psi$, this projection is the disk of radius $\lambda_j$ in $\C$, centered at $a_j$. Since $\psi$ and its suspensions are affine mappings, for the $\theta$-extension the image is the disk of radius
$$
\theta\lambda_j=\theta r_j\left [1-\frac{1}{\beta^2} \right ]^{-\frac{1}{2}}\le \theta \frac{|a_j|}{\zeta}\left [1-\frac{1}{\beta^2} \right ]^{-\frac{1}{2}}=\frac{1}{2}|a_j|,
$$
and hence this disk does not touch $0\in \C$. This completes the proof of Proposition \ref{prp:doubl.cov.susp2}.
\end{proof}

\begin{proof} [Proof of Proposition \ref{prp:doubl.cov.S.eta}]
We have to construct, for $\gamma\ge 2$, a $\gamma$ covering $\U$ of $Q^\eta_n$ with the complexity $\kappa(\U)$ at most $(9\gamma^n)^n(\log ({9\gamma^n}/{\eta}))^n$. We fix $n$, and proceed by induction on the dimension $l$ of $Q^\eta_l$: for $l=1,2,\ldots,n$ we show the existence of a $\gamma^{n-l+1}$-covering $\U_l$ of $Q^\eta_l$ with $\kappa(\U_l)\le (9\gamma^n)^l(\log ({9\gamma^n}/{\eta}))^l$.

In dimension $l = 1$ the result is a partial case of \cite[Theorem 2.2]{FY17}: for any $\zeta >1$ the required $\zeta$-covering of $D_1\setminus D_\eta$ in $D_1\setminus \{0\}$ consists of $3\zeta\log(3\zeta/\eta)$ ``Whitney's disks'', accumulating to the origin. To start the induction, we put in this theorem $\zeta = \gamma^n$, and get a $\gamma^n$-covering of $D_1\setminus D_\eta$ in $D_1\setminus \{0\}$ consisting of less than $9\gamma^n\log(9\gamma^n/\eta)$ Whitney's disks.

Assume that the required $\gamma^{n-l+1}$-covering $\U$ of $Q_l^\eta$ has been constructed in dimension $l$ with $\kappa(\U) \le (9\gamma^n)^l(\log ({9\gamma^n}/{\eta}))^l$. We produce the required covering of $Q_{l+1}^\eta$ in dimension $l+1$, using the fact that $Q_{l+1}^\eta = Q_l^\eta \times \{D_1\setminus D_\eta\}$. We use the suspension construction, as developed above, and apply Proposition \ref{prp:doubl.cov.susp2} with $Y=\C^l\setminus Z_l$, $G=Q_l^\eta$, $\delta=\eta$, $\mu=\gamma^{n-l+1}$, and $\beta = \gamma$. Thus, $\theta=\frac{\mu}{\beta}=\gamma^{n-l}$, and we obtain a $\gamma^{n-l}$-covering $\U_{l+1}$ of $Q_{l+1}^\eta = Q_l^\eta \times \{D_1\setminus D_\eta\}$ in $\C^{l+1}\setminus Z_{l+1}$ with the complexity $\kappa(\U_{l+1})$ not exceeding $3\zeta \log(\frac{3\zeta}{\eta}) \kappa(\U_{l})$, where
$$
\zeta =\frac{2\mu}{\beta}\left [1-\frac{1}{\beta^2} \right ]^{-\frac{1}{2}}=2\gamma^{n-l}\left [1-\frac{1}{\beta^2} \right ]^{-\frac{1}{2}}\le 2\gamma^{n-l}(\frac{3}{4})^{-\frac{1}{2}}\le 3\gamma^{n-l},
$$
since by assumptions $\beta=\gamma\ge 2$. Therefore, we have
$$
\kappa(\U_{l+1})\le 9\gamma^{n-l}\log(\frac{9\gamma^{n-l}}{\eta}) \kappa(\U_{l})\le 9\gamma^{n-l}\log(\frac{9\gamma^{n-l}}{\eta})(9\gamma^n)^l(\log ({9\gamma^n}/{\eta}))^l\le
$$
$$
\le (9\gamma^n)^{l+1}(\log ({9\gamma^n}/{\eta}))^{l+1}.
$$
This completes the induction step. For $l=n$ we get a $\gamma^{n-n+1}=\gamma$-covering $\U=\U_n$ of $Q_n^\eta$ with the complexity satisfying
$$
\kappa (\U)=\kappa (\U_n)\le (9\gamma^n)^n(\log ({9\gamma^n}/{\eta}))^n,
$$
thus completing the proof of Proposition \ref{prp:doubl.cov.S.eta}.
\end{proof}

\subsection{Proof of Theorem \ref{thm-regular}} \label{sec:cov.hypersurface}

As it was explained above, in order to prove our second main result it is sufficient to construct a required doubling covering ``locally'', in each coordinate neighborhood $W_y$ provided by Theorem \ref{Res.Sing}. As above, we assume that in local coordinates $x_1,\ldots,x_n$ the neighborhood $W_y$ is defined by $|x_i|\le 1$, $i=1,\ldots,n$, while the polynomial $P\circ\sigma$ takes a form

\be \label{eq:monomial5}
P\circ\sigma(x)=U(x) x^\alpha , \text \ {with} \ \frac{1}{C_y}\le |U(x)|\le C_y, \ x \in W_y.
\ee
Let $\tilde W_y$ be defined by $|x_1|\le \frac{1}{4}$, $|x_i|\le 1$, $i=2,\ldots,n$. We will produce, for a regular value $c>0$, a doubling covering of $Y_c\cap \tilde W_y$, with $Y_c$ defined in $W_y$ by the equation $P\circ\sigma(x)=c$.

We present the hypersurface $Y_c$ as the graph $x_1 = g(x_2,\ldots,x_n)$ over $Q_{n-1}^\eta$, for an appropriate $\eta>0$. The function $g$ is a multivalued (more accurately, $\alpha_1$-valued) function and we show that all its branches are regular. Finally, we use Proposition \ref{prp:doubl.cov.S.eta} to construct a doubling covering $\U$ of $Q_{n-1}^\eta$ and hence, of $\Omega$ and extend it to the required covering of $Y_c$, composing the charts in $\U$ with $g$.

Now, we present the proof in detail.

\begin{lem}
Let $\alpha_0 = \min_{i = 1}^n \alpha_i \ge 1$. Let $x = (x_1,x_2,\ldots,x_n)\in Y_c \cap W_y$. Then, for any $1\le j \le n$, we have
$$
|x_j| \ge (\frac{|c|}{C_y})^{\frac{1}{\alpha_0}} .
$$
In particular, the projection $\Omega$ of $Y_c \cap W_y$ onto the subspace $\C^{n-1}$ of the points $\bar x = (x_2,\ldots,x_n)$ in $\C^n$ is contained in $Q_{n-1}^\eta$, for $\eta = (\frac{|c|}{C_y})^{\frac{1}{\alpha_0}}$.
\end{lem}

\begin{proof}
We have
$$
|x_j|^{\alpha_j} = \frac{|c|}{|U(x)| \prod_{i\ne j} |x_i|^{\alpha_i}} \ge \frac{|c|}{C_y},
$$
since, by assumptions, for any $l$ we have $|x_l| \le 1$, and $\frac{1}{C_y}\le |U(x)|\le C_y$, $x \in W_y$. Therefore
$$
|x_j| \ge (\frac{|c|}{C_y})^{\frac{1}{\alpha_j}} \ge (\frac{|c|}{C_y})^{\frac{1}{\alpha_0}}=\eta.
$$
\end{proof}

Next we use Proposition \ref{prp:doubl.cov.S.eta} to construct a $2$-doubling covering $\U$ of $Q_{n-1}^\eta$. In order to extend $\U$ to the required covering $\bar \U$ of $Y_c$, we show, using the implicit function theorem, that the equation \eqref{eq:monomial5} of the hypersurface $Y_c$ locally defines each branch of $x_1 = g(x_2,\ldots,x_n)$ as a regular function. Then, we compose the charts in $\U$ with $g$, in order to get the charts of $\bar \U$.

Let us fix $\bar x = (x_2,\ldots,x_n)\in \Omega$, and consider a function of one variable $v(x_1): = U(x_1,\bar x)x_1^{\alpha_1}\bar x^{\bar \alpha}$, for $\bar \alpha = (\alpha_2,\ldots,\alpha_n)$.

\begin{lem} \label{lem:deriv.g}
For $0<|x_1|\le \frac{1}{4}$ we have $\frac{\partial}{\partial x_1}(U(x_1,\bar x)x_1^{\alpha_1}\bar x^{\bar \alpha}) = \frac{d}{dx_1}v(x_1)\ne 0$.
\end{lem}

\begin{proof}
We have
\begin{align*}
\frac{\partial}{\partial x_1}(U(x_1,\bar x)x_1^{\alpha_1}\bar x^{\bar \alpha}) & = \frac{\partial}{\partial x_1}U(x_1,\bar x)\cdot x_1^{\alpha_1}\bar x^{\bar \alpha}+\alpha_1 U(x_1,\bar x)x_1^{\alpha_1-1}\bar x^{\bar \alpha} \\
& = x_1^{\alpha_1-1}\bar x^{\bar \alpha}(\frac{\partial}{\partial x_1}U(x_1,\bar x)\cdot x_1+\alpha_1 U(x_1,\bar x)).
\end{align*}
Since by assumptions we have $\frac{1}{2} C_y \le |U(x)| \le 2 C_y$, by Cauchy formula we conclude that $|\frac{\partial}{\partial x_1}(U(x_1,\bar x))| \le C_y$. Consequently, for $0<|x_1| \le \frac{1}{4}$ we have
$$
\frac{\partial}{\partial x_1}U(x_1,\bar x))x_1+\alpha_1 U(x_1,\bar x)\ne 0.
$$
and hence $\frac{\partial}{\partial x_1}(U(x_1,\bar x)x_1^{\alpha_1}\bar x^{\bar \alpha}) = \frac{d}{dx_1}v(x_1)\ne 0$.
\end{proof}

Lemma \ref{lem:deriv.g}, combined with the implicit function theorem, shows that equation $P\circ\sigma(x)=c$ of the hypersurface $Y_c$ locally defines each branch of $x_1 = g(x_2,\ldots,x_n)$ as a regular function. Consequently, for any chart $U_j\in \U$, for any $\bar x \in \U$ and for any choice of the branch $g(\bar x)$ at $\bar x$ there is a unique analytic continuation of $g$ to the entire chart $U_j$. Indeed, using a local regularity of the chosen branch of $g$, and extending it along the straight segments from $\bar x$ to any other point of the ellipsoid $U_j$, (and, in fact, to $2 U_j$) we obtain the required continuation of $g$ to the entire chart $U_j$ and to its $2$-concentric extension. The corresponding chart $\bar U_j$ in $\bar \U$ is obtained as the composition $g\circ U_j$. The entire collection of the charts in $\bar \U$ is obtained as we compose all the charts $U_j\in \U$ with all the branches of $g$ over $U_j$.

Clearly, the charts in $\bar \U$ are $2$-doubling charts in $Y_c$. The complexity $\kappa(\bar \U)$, i.e. the number of the charts, is equal to $\alpha_1\times \kappa(\U)$. Now, choosing the neighborhoods $W_y$ as in the proof of Theorem \ref{thm-complement}, and applying the arguments above, as well as Proposition \ref{prp:doubl.cov.S.eta}, to each $W_y$, we obtain the required complexity bound. This completes the proof of Theorem \ref{thm-regular}. $\square$

\section{Analytic parameterizations of real semi-algebraic, subanalytic, and power-subanalytic sets} \label{sec:real}

In this section we treat analogues of Theorems \ref{thm-complement} and \ref{thm-regular} for real semi-algebraic, subanalytic, and power-subanalytic sets (see Theorems \ref{thm:semi-alg-a-chart}, \ref{thm:fibers-semi-alg-a-chart}, \ref{thm:semi-a-chart:g} and \ref{thm:fibers-a-chart:g:g}).
These notions of sets generalize the ones of globally subanalytic sets and of real semi-algebraic sets and are recalled below. The main idea behind the proofs of these two theorems is similar to the complex reduction from the previous section to the monomial case, this time not exactly by resolving the singularities, but, by using a pre-parameterization result from \cite{CPW}, based on preparation of power-subanalytic functions from \cite{M06}, and a rectilinear variant of preparation for subanalytic functions from \cite{CM13}. All these mentioned results are incarnations on the reals of results related to both Weierstrass preparation and resolution of singularities. In fact, we give a refined pre-parameterization which combines the mentioned results from \cite{CPW} and \cite{CM13}, see Theorem \ref{recti-pre-param}.

\subsection{Analytic parameterizations} \label{sec:a-b-m}

We define the following generalization of real semi-algebraic sets and functions, as an example to which the results in this section apply.
Call a set $X\subset \RR^n$ power-semi-algebraic if it is given by a finite Boolean combination of conditions on $x\in\RR^n$ of the form
$$
0 < p(x,(x_{i_1}^2)^{r_{i_1}},\ldots, (x_{i_s}^2)^{r_{i_s}})
$$
for some polynomials $p$ with coefficients in $\RR$, some integer $s\geq 0$, and some positive real numbers $r_{i_j}$ for some $i_j\in\{1,\ldots,n\}$ with $j=1,\ldots,s$. Call a function $f:X\to Y$ power-semi-algebraic if $X$, $Y$, and the graph of $f$ are power-semi-algebraic sets. By the complexity of such a Boolean combination, we mean the tuple consisting of the number of involved polynomials and for each involved polynomial the total degree, the number of variables of the polynomial (namely $n+s$ for $p$ as above), and the real numbers $r_{i_j}$. If no real exponents occur, (namely $s=0$ in the occurring polynomials as $p$ above), then one says semi-algebraic instead of power-semi-algebraic.

As a second and richer setting, let us define power-subanalytic sets, as generalization of globally subanalytic subsets of $\RR^n$.
Call a function $f:\RR^n\to \RR^m$ a restricted analytic map if its restriction to $[0,1]^n$ is analytic (in the above sense), and, the restriction of $f$ to the complement of $[0,1]^n$ in $\RR^n$ is identically zero. (Note that no continuity of $f$ is required on the boundary of $[0,1]^n$.)
Call a function $\RR^n\to \RR^m$ a power-basic function if it is a composition $f_s\circ\ldots\circ f_1$ for some $s$, where each $f_i$ is either a power-semi-algebraic map or a restricted analytic map.
Call a set $X\subset \RR^n$ power-subanalytic if it is given by a finite Boolean combination of conditions on $x\in\RR^n$ of the form
$$
B(x) >0
$$
for some power-basic functions $B$. Call a function $f:X\to Y$ power-subanalytic if $X$, $Y$, and the graph of $f$ are power-subanalytic sets. (Sometimes one says $\Ranp$-definable or $\Ran^\RR$-definable instead of power-subanalytic.) When moreover the involved power-semi-algebraic maps are semi-algebraic, then one calls the sets and functions globally subanalytic (or, in short, subanalytic).
By D. Miller's work \cite{M06}, the power-subanalytic sets form an o-minimal structure, the subject of \cite{vdD98}.
By the dimension of a (nonempty) power-subanalytic set $X\subset \RR^n$, we mean the maximum integer $\ell\geq 0$ taken over all linear maps $L:\RR^n\to \RR^\ell$ such that $L(X)$ has nonempty topological interior (this has good properties coming from o-minimality, see e.g. \cite{vdD98}).

 We now come to our two main results on parameterizations of power-subanalytic sets, resp. of subanalytic sets.

\begin{thm} \label{thm:semi-a-chart:g} Let $n$ and $m$ be positive integers.
Let $T$ and $\cX\subset T\times I^n$ be power-subanalytic sets such that for each $t\in T$, the fiber $X_t:=\{x\in I^n \mid (t,x)\in \cX\}$ has dimension $m$.
Then, there exist a power-subanalytic set $\cS\subset T\times I^n$ such that each fiber $S_t:=\{x\in I^n \mid (t,x)\in\cS\}$ has dimension $<m$ and a constant $C$ depending only on $\cX$ such that the following holds. For each $\delta>0$ with $\delta\leq 1$
there are power-subanalytic functions
$$
\psi_i:T\times I^m \to \RR^n \mbox{ for } i=1,\ldots, \kappa(\delta)
$$
with
$$
 \kappa(\delta)\leq C(\log 1/\delta)^m
$$
such that, for each $t\in T$, the maps
$$
\psi_{i,t}:I^m \to \RR^n:x\mapsto \psi_i(t,x)
$$
are a-charts and
$$
X_t\setminus S_{t,\delta}\subset \bigcup_{i=1}^{\kappa(\delta)} \psi_{i,t}(I^m),
$$
where $S_{t,\delta}$ is the $\delta$-neighborhood of  $S_t$ in $\RR^n$.
\end{thm}

\begin{thm} \label{thm:fibers-a-chart:g:g} Let $\ell,m,n,k$ be positive integers.
Let $T\subset I^k$ and $\cX\subset T\times I^n$ be subanalytic sets such that, for each $t\in T$, the fiber $X_t:=\{x\in I^n\mid (t,x)\in \cX\}$ has dimension $m$.
Then, there exist a constant $C$ and a subanalytic set $S\subset I^k$ of dimension less than the dimension of $T$ such that the following holds for each $\delta>0$ with $\delta\leq 1$.
There are subanalytic functions
$$
\psi_i:T\times I^{m} \to I^n \mbox{ for } i=1,\ldots, \kappa(\delta)
$$
with
$$
\kappa(\delta)\leq C(\log 1/\delta)^{m}
$$
such that for any $t\in T\setminus S_\delta$ the functions
$$
\psi_{i,t}:I^{m}\to \RR^n : x\mapsto \psi_i(t,x)
$$
are a-charts and
$$
X_t \subset \bigcup_{i=1}^{\kappa(\delta)} \psi_{i,t}(I^m),
$$
where $S_\delta$ the $\delta$-neighborhood  of  $S$.
\end{thm}

In Theorem \ref{thm:semi-a-chart:g}, small tubes are removed of the sets which are to be parameterized, where as in Theorem \ref{thm:fibers-a-chart:g:g}, a small tube is removed from the parameter space, thus leaving out a small portion of the family members. We leave to the reader to formulate the special case of Theorem \ref{thm:fibers-a-chart:g:g} with level sets, similar as in Theorem \ref{thm-regular}.

For the proofs of Theorems \ref{thm:semi-a-chart:g} and \ref{thm:fibers-a-chart:g:g} we will use freely basic properties of a-charts given in \cite{Y91} (for acu's). Note that Theorem \ref{thm:semi-alg-a-chart} is a special case of Theorem \ref{thm:semi-a-chart:g}.
We give a slight refinement of Theorem \ref{thm:fibers-a-chart:g:g} at the end of the paper, in Section \ref{sec:weier:sys}, from which Theorem \ref{thm:fibers-semi-alg-a-chart} follows as a special case.

\subsection{Pre-parameterization and the proof of Theorem \ref{thm:semi-a-chart:g}}
 \label{sec:pre-par-proof}

We recall the notions of bounded-monomial functions and of a-b-m functions from \cite{CPW} as the following Definitions \ref{defn:b-m} and \ref{defn:prepared}.

\begin{defin}[bounded-monomial functions] \label{defn:b-m}
Let $U$ be a subset of $(0,1)^m$. A function $b:U\to \RR$ with bounded range is called bounded-monomial if either $b$ is identically zero, or, $b$ is
of the form
$$
x \mapsto x^\mu := \prod_{i=1}^m x_i^{\mu_{i}}
$$
for some $\mu_{i}$ in $\RR$ and $\mu=(\mu_i)_i$. We say that only integer exponents appear in the bounded-monomial function $b$ if moreover $\mu\in \mathbb{Z}^m$ (including the case that $b$ is identically zero). A map $U\to \RR^n$ is called bounded-monomial if all of its component functions are, and similarly for the appearance of only integer exponents.
\end{defin}

\begin{defin}[a-b-m functions] \label{defn:prepared}
Let $U$ be a subset of $(0,1)^m$.
A function $f:U\to \RR$ is called a-b-m, in full analytic-bounded-monomial, if it is of the form
$$
f(x) = b_j(x) F( b_1(x),\ldots, b_s(x) )
$$
for some bounded-monomial map $b:U\to \RR^s$ for some $s$ and for some nonvanishing analytic function $F:V\to \RR$, where $V$ is an open neighborhood of $\overline{b(U)}$, the topological closure of $b(U)$ in $\RR^s$, and where $j$ lies in $\{1,\ldots,s\}$. We call the map $b$ an associated bounded-monomial map of $f$.

Finally, call a map $f:U\to \RR^n$ a-b-m, with associated bounded-monomial map $b$, if all its component functions are (namely, each $f_i$ is a-b-m, and, $b$ is an associated bounded-monomial map for each $f_i$).
\end{defin}

The a-b-m functions with an associated bounded-monomial map $b$ such that moreover $b$ has bounded $C^1$-norm have particularly nice properties as illustrated by their use in \cite{CPW} and in the proofs of Theorems \ref{thm:semi-a-chart:g} and \ref{thm:fibers-a-chart:g:g}.

\begin{defin}[Cells and their walls]
A power-subanalytic subset $C\subset \RR^n$ is called a cell, if
$$
C = \{x\in \RR^n\mid \wedge_{i=1}^n\, \alpha_i(x_{<i}) \sq_{i1} x_i \sq_{i2} \beta_i(x_{<i}) \}
$$
for some continuous power-subanalytic functions $\alpha_i$ and $\beta_i$ with $\alpha_i < \beta_i$, $x_{<i} = (x_1,\ldots,x_{i-1})$, and with $\sq_{i1}$ either $=$, $<$, or no condition, and with $\sq_{i2}$ either $<$ or no condition, with the conventions that $\sq_{i2}$ is no condition if $\sq_{i1}$ is equality. If $\sq_{i1}$ is $=$ or $<$ then we call $\alpha_i$ a wall of $C$. Likewise, if $\sq_{i2}$ is $<$ then we also call $\beta_i$ a wall of $C$.
\end{defin}

We can now recall the pre-parameterization result from \cite{CPW} that we use to prove Theorems \ref{thm:semi-a-chart:g} and \ref{thm:fibers-a-chart:g:g}. The boundedness of the $C^1$-norms in item \ref{pre4} is a key property (without this boundedness, the result would be much more easy to prove). Note that the triangularity property from \ref{pre3} allows one to use the result uniformly in family settings, and this is indeed exploited in this way below as well as in \cite{CPW}.

\begin{thm}[Pre-parameterization, \cite{CPW}] \label{pre-param}
Let $X\subset (0,1)^{n}$ be power-subanalytic, and suppose that $X$ is the graph of a power-subanalytic function $f: U\to (0,1)^{n-m}$ for some $m\geq 0$ and open set $U\subset (0,1)^m$. Then, there exist finitely many power-subanalytic maps
$$
\varphi_i:U_i\to X
$$
such that the following hold
\begin{enumerate}[label = \textup{(\arabic*)}, ref = (\arabic*)]
\item \label{pre1} $\bigcup_i \varphi_i(U_i)=X$.

\item \label{pre2} Each $U_i$ is an open cell in $(0,1)^m$. 

\item \label{pre3} Each $\varphi_i$ is a triangular map, in the sense that for each $j\leq m$ there is a unique map $\Pi_{<j}(U_i)\to \Pi_{<j}(X)$ making a commutative diagram with $\varphi_i$ and the projection maps $X\to
 \Pi_{<j}(X)= \Pi_{<j}(U)$ and $U_i\to \Pi_{<j}(U_i)$, with in both cases $\Pi_{<j}$ the projection to the first $j-1$ coordinates.

\item \label{pre4}
For each $i$, the map $\varphi_i$ and the walls $\alpha$ of $U_i$
are a-b-m with an associated bounded-monomial map $b_i$ such that $b_i$ has bounded $C^1$-norm.
\end{enumerate}
\end{thm}

We can now give the proof of Theorem \ref{thm:semi-a-chart:g}.

\begin{proof}[Proof of Theorem \ref{thm:semi-a-chart:g}]
By transforming $T$ if necessary and by working piecewise on $\cX$, we may suppose that $T\subset (0,1)^{k}$ for some $k$, that $T$ is open, and that $\cX$ equals the graph of a power-subanalytic function $f:\cU\to \RR^{n-m}$ for some open $\cU\subset T\times (0,1)^m$ (indeed, these are typical manipulations involving basic finiteness properties of o-minimality, see \cite{vdD98}). Apply the pre-parameterization result Theorem \ref{pre-param} to $\cX$. Up to another transformation of $T$ and working piecewise on $T$, we reduce to the case that we have finitely many power-subanalytic maps
$$
\varphi_i:\cU_i\to \cX
$$
such that $\cU_i\subset T\times (0,1)^m$ is definable and open, and such that $\varphi_i$ and the walls of $\cU_i$ are a-b-m maps with associated bounded monomial map $b_i$ with bounded $C^1$-norm. Clearly, by their special nature, the $\varphi_i$ have a unique continuous extention $\overline \varphi_i:\overline\cU_i \to \RR^{k+n}$ to the topological closure $\overline\cU_i$ of $\cU_i$ in $\RR^{k+m}$. This extension is power-subanalytic, since clearly definable.
Let $\cS$ be the union over $i$ of the sets $\overline\varphi_i ( \overline\cU_i\setminus \cU_i )$, namely the images of the boundaries. Then, $\cS$ has dimension less than $m$ (see the dimension theory for o-minimal structures explained in \cite{vdD98}). Also, by the special form of the maps $\varphi_i$, there is a constant $c\geq 1$ such that for each $t$ and each $i$, the map $\varphi_{i,t}$ is Lipschitz-continuous with Lipschitz constant $c$, where the metric is the supremum norm.
Now, fix $\delta>0$.
For each $t\in T$, write
$$
\varphi_{i,t}:U_{i,t}\to X_t:x\mapsto \varphi_i(t,x)
$$
with
$$
U_{i,t}:=\{x\mid (t,x)\in \cU_i\}.
$$
Further, write
$$
\varphi_{i,t,\delta}
$$
for the restriction of $\varphi_{i,t}$
to
$$
U_{i,t}\cap (\delta/c,1)^m.
$$
Then, by the mentioned Lipschitz continuity with Lipschitz constant $c$, the union over $i$ of the images of the $\varphi_{i,t,\delta}$ contains $X_t\setminus S_{t,\delta}$. We claim that there is a constant $C\geq 1$ such that for each $t$ and each $i$, the graph of $\varphi_{i,t,\delta}$ can be covered by no more than $ C \log(1/\delta)^m$ a-charts. This can be seen as follows. Since $b_i$ is bounded-monomial and by Lemma \ref{lem:b-m-a-chart}, for each $i$ there is $C_i$ such that for each $t\in T$, the graph of
$$
b_{i,t}:(\delta/c,1)^m \cap U_{i,t} \to \RR^n:x\mapsto b_i(t,x)
$$
can be covered by no more than $C_i \log ( 1/\delta)^m$ a-charts. Now, the claim and the theorem follow from properties for covering compositions by a-charts from \cite{Y91, Y08}.
\end{proof}

The following lemma treats the case of monomial functions with real exponents and bounded range.

\begin{lem} \label{lem:b-m-a-chart}
Given $\mu\in\RR^m$, there exists $C>0$ such that the following holds.
Let $U\subset (0,1)^m$ be open and let $b:U\to (0,1)$ be a map of the form
$$
x\mapsto a x^\mu
$$
for some real $a>0$ and some $\mu\in\RR^m$. For each $\varepsilon>0$ with $\varepsilon<1/2$ let $b_\varepsilon$ be the restriction of $b$ to $U\cap (\varepsilon,1)^m$.
Then, there are $N_\varepsilon$ many a-charts $\varphi_i:I^m\to \RR^{m+1}$ with $N_\varepsilon \leq C\log(1/\varepsilon)^m$ and such that the graph of $b_\varepsilon$ is contained in the union of the sets $\varphi_i(I^m)$.
\end{lem}

\begin{proof}[Proof of Lemma \ref{lem:b-m-a-chart}]
For any $A>0$, let $U_A$ be the set of $x\in \RR_{>0}^m$ such that $a x^\mu<A$ with $\RR_{>0}$ the set of positive real numbers, and let $b_A$ be the map $x\mapsto a x^\mu$ on $U_A$. Write $U_{1,\varepsilon} := U_1\cap (\varepsilon,1)^m$ and let $b_{1,\varepsilon}$ be the map $x\mapsto a x^\mu$ on $U_{1,\varepsilon}$.
We will in fact prove slightly more than the lemma: we will cover the graph of $b_{1,\varepsilon}$ by $N_\varepsilon$ a-charts going into the graph of $b_A$ with $A=2^{M}$ and $M=\sum_i |\mu_i|$ and with $N_\varepsilon \leq C\log(1/\varepsilon)^m$ for $C$ depending only on $\mu$. For any
$z\in (0,1)$, let $I_z$ be the open interval $(z/2,3z/2)$ in $\RR_{>0}$. Choose any $y\in U_{1,\varepsilon}$. Then, by construction, the set
$$
B_y := \prod_{i=1}^m I_{y_i}
$$
is contained in $U_A$. Moreover, $(\varepsilon,1)^m$ and hence also $U_{1,\varepsilon}$ can clearly be covered by $C_1\log(1/\varepsilon)^m$ many sets of the form $B_y$ with $y$ in $U_{1,\varepsilon}$, with $C_1$ a constant depending only on $\mu$. Finally, we show for any $y\in U_{1,\varepsilon}$ that the graph of the map
$$
{\mathbf{b}}_y: B_y \to\RR: x\mapsto a x^\mu
$$
can be covered by no more than $C_2$ a-charts, with $C_2$ a constant depending only on $\mu$.
But this can be seen by composing ${\mathbf{b}}_y$ with the map
$$
(-C_3,C_3)^m\to B_{y} : z \mapsto (y_1 + z_1y_1/2C_3,\ldots, y_m + z_m y_m/2C_3)
$$ for some sufficiently large $C_3\geq 1$ depending only on $\mu$, and by taking the Taylor series around $0$ of the composition. Indeed, the estimates on the Taylor coefficients are easy to obtain.
\end{proof}

Note that $b/a$ is a bounded-monomial function with $a$ and $b$ as in Lemma \ref{lem:b-m-a-chart}, but, when $a\not=1$, then $b$ itself is not bounded-monomial.
 
\subsection{Rectilinear pre-parameterization and the proof of Theorem \ref{thm:fibers-a-chart:g:g}} \label{subs:recti}

In the subanalytic case, we can give a refinement of the pre-parameterization result of \cite{CPW}, by combining with the notion and results about rectilinear cells of Theorem 1.5 of \cite{CM13}. This will be used to prove Theorem \ref{thm:fibers-a-chart:g:g}. We leave the discovery of a variant of Theorem \ref{recti-pre-param} for power-subanalytic sets to the future.

\begin{defin}[Rectilinear cells] \label{defn:recti}
Let $m$ and $\ell$ be positive integers with $0\leq \ell\leq m$.
An open cell $A\subset (0,1)^m$ is called $\ell$-rectilinear if it is of the form $B\times (0,1)^{m-\ell}$, where $B$ is an open cell satisfying $\overline B\subset (0,1)^\ell$, where $\overline B$ is the topological closure of $B$ in $\RR^\ell$.
\end{defin}

The refinement given by the following variant of Theorem \ref{pre-param} lies in the property that the cells $U_{i,t}$ are rectilinear in \ref{pre2r}, and, the appearance of only integer exponents in \ref{pre4r}.

\begin{thm}[Rectilinear pre-parameterization] \label{recti-pre-param}
Let $T\subset (0,1)^k$ and $\cX\subset T\times (0,1)^{n}$ be subanalytic, and suppose that $\cX$ is the graph of a subanalytic function $f: \cU \to (0,1)^{n-m}$ for some $m\geq 0$ and some subanalytic $\cU\subset T\times (0,1)^m$ such that $U_t:=\{x\in (0,1)^m \mid (t,x)\in \cU\}$ is nonempty and open in $(0,1)^{m}$ for each $t\in T$. Then, there exist finitely many subanalytic maps
$$
\varphi_{i}:\cU_i\to \cX
$$
for some subanalytic sets $\cU_i\subset (0,1)^k\times (0,1)^m$ and integers $\ell_i\geq 0$
such that the following hold
\begin{enumerate}[label = \textup{(\arabic*)}, ref = (\arabic*)]
\item \label{pre1r} $\bigcup_i \varphi_i(\cU_i)=\cX$.

\item \label{pre2r} Each $\cU_{i}$ is an open cell in $(0,1)^{k+m}$, and,
for each $t\in (0,1)^k$, the set $U_{i,t}:=\{x\mid (t,x)\in \cU_i\}$ (when nonempty) is an $\ell_i$-rectilinear open cell in $(0,1)^m$.

\item \label{pre3r} Each $\varphi_i$ is a triangular map, in the sense that for each $j\leq k+m$ there is a unique map $\Pi_{<j}(\cU_i)\to \Pi_{<j}(\cX)$ making a commutative diagram with $\varphi_i$ and the projection maps $\cX\to
 \Pi_{<j}(\cX)= \Pi_{<j}(\cU)$ and $\cU_i\to \Pi_{<j}(\cU_i)$, with in both cases $\Pi_{<j}$ the projection to the first $j-1$ coordinates.

\item \label{pre4r} 
For each $i$, the map $\varphi_i$ and the walls $\alpha$ of $\cU_i$ 
are a-b-m with an associated bounded-monomial map $b_i$ such that $b_i$ has bounded $C^1$-norm and such that only integer exponents appear in $b_i$. 
\end{enumerate}
\end{thm}

The proof is similar to the proof of the Pre-parameterization result of \cite{CPW} where moreover Theorem 1.5 of \cite{CM13} is used to make the initial situation already rectilinear.

\begin{proof}[Proof of Theorem \ref{recti-pre-param}]
By Theorem 1.5 of \cite{CM13} we may suppose that $\cX$ is the graph of a subanalytic function $f:\cU\subset T\times (0,1)^{m}\to (0,1)^{n-m}$ such that $\cU$ is an open cell in $(0,1)^{k+m}$ and such that moreover $X_t$ is $\ell$-rectilinear for each $t\in T$ and some $\ell\geq 0$ independent from $t$. Moreover, by the same theorem of \cite{CM13}, we may suppose that $f$ and all the walls of $\cU$ are a-b-m with associated bounded-monomial map $b$ with only integer exponents.

We now show by induction on $m$ that from this situation on, up to some parts with a lower value for $m$ (which can be treated by induction on $m$), we can partition $\cX$ into finitely many parts $\cX_i$ each of which can be reparamaterized by maps $\varphi_i$ as in the theorem with moreover $\ell_i=\ell$.
Suppose $m\geq 1$. The map $b$ is $C^1$, since it is bounded-monomial.
By a classical technique (with inverse functions) we will reduce to the case that furthermore $|\partial b_{j}/\partial x_m|$ is at most $1$ for each component function $b_{j}$ of $b$.
First note that if $\cU$ is of the form $\cU'\times (0,1)$ for some $\cU'\subset (0,1)^{k+m-1}$ (that is, we are in the case that $\ell<m$), then $|\partial b_{j}/\partial x_m|$ is already bounded for each component function $b_{j}$ of $b$ since $b$ is bounded-monomial with only integer exponents. (Indeed, the boundedness of $b$ forces the exponents of $x_m$ to be nonnegative integers.)
In the other case that $\ell=m$, we proceed as follows.
Up to partitioning $\cU$ into finitely many definable pieces and neglecting pieces where $U_t$ is of lower dimension by induction on $m$, we may suppose that there is $j$ such that $|\partial b_{j}/\partial x_m|$ is maximal on $\cU$, in the sense that
\begin{equation} \label{eq:bjbj'}
|\partial b_{j}(x)/\partial x_m|\geq |\partial b_{j'}(x)/\partial x_m|
\end{equation}
on $\cU$ for any $j'$. This partitioning based on conditions of the form (\ref{eq:bjbj'}) preserves the $m$-rectilinear form, as well as the fact that the walls are a-b-m, even with the very same bounded-monomial map $b$ with only integer exponents.
Similarly, for this $j$ we may furthermore suppose that either $|\partial b_{j}/\partial x_m|\leq 1$ on $\cU$, or, that $|\partial b_{j}/\partial x_m| > 1$ on $\cU$. In the first case, we have what we want at this point.
In the second case, we note that the function sending $x_m$ to $b_{j}(t,x_{<m},x_m)$ is injective, for each choice of $(t,x_{<m}) = (t_1,\ldots,t_k,x_1,\ldots,x_{m-1})$, since $b$ is bounded-monomial.
Up to replacing $\cX$ by the graph of the function sending $(t,x_{<m},b_j(t,x)/N)$ to $(t,x, f(t,x))$, where $(t,x)$ is in $\cU$ and $N>0$ is sufficiently large, 
we may thus suppose (by the chain rule) that we are in the first case, namely, that $|\partial b_{j}/\partial x_m|\leq 1$ on $\cU$. Note that this change of variables preserves $\ell=m$. We have thus reduced to the case that furthermore $|\partial b_{j}/\partial x_m|$ is at most $1$ for each component function $b_{j}$ of $b$.

\par

We still need to show that we can ensure that the $C^1$-norm of $b$ is bounded.
Since $b$ is a bounded-monomial map, there is $N\geq 1$ such that
$$
|b_{j}/N|<1-\varepsilon,\mbox{ and } |\frac{1}{N}\frac{\partial b_{j}}{\partial x_m}|<1-\varepsilon
$$
for each component function $b_{j}$ of $b$ and some $\varepsilon>0$.
For each wall $\alpha$ of $\cU$ bounding $x_m$, and with $h$ being either $b/N$ or $(1/N)\partial b / \partial x_m $, 
let $h_{\alpha}$ be the map
\begin{equation} \label{halpha}
h_{\alpha} : \Pi_{<m}(\cU )\to (-1,1)^s: (t,x_{<m}) \mapsto \lim_{x_m\to \alpha(t, x_{<m} ) } h(t,x_{<m},x_m),
\end{equation}
where $\Pi_{<m}(\cU)$ is the image of $\cU$ under the coordinate projection $\Pi_{<m}$ sending $(t,x)$ to $(t,x_{<m})$.
This limit always exists by the definition of bounded-monomial maps, and, moreover, $Nh_{\alpha}$ is a bounded-monomial map itself.
Let $\cG$ be the collection of functions on $\Pi_{<m}(\cU)$ consisting of
the component functions of the maps $h_\alpha$ from (\ref{halpha}) and the walls $\alpha$ of $\cU$ bounding $x_m$.
Consider the map $F$ whose component functions are the maps $|g|$ for those $g$ in $\cG$ which are not identically zero.
Apply the induction hypothesis, for $m-1$ instead of $m$ and with $\min (\ell,m-1)$ instead of $\ell$, to the graph of $F$ instead of the graph of $f$,
to find a finite collection of maps $\psi_{\tau }:V_{\tau }\to \textrm{Graph}(F)\setminus \cS$ satisfying properties \ref{pre1r}, \ref{pre2r}, \ref{pre3r}, and \ref{pre4r}, with $\textrm{Graph}(F)\setminus \cS$ in the role of $\cX$, with $\ell_\tau=\min (\ell,m-1)$ for each $\tau$, with associated bounded-monomial maps $c_{\tau}$, and with $S_t$ of dimension smaller than $m-1$ for each $t$.
Using these newly obtained maps $\psi_{\tau }$ we easily get finitely many maps $\varphi_\tau$ with properties \ref{pre1r}, \ref{pre2r}, \ref{pre3r}, and \ref{pre4r} for $\cX\setminus \cS'$ where $\ell_i=\ell$ for each $i$ and for some $\cS'$ where $S'_t:=\{x \mid (t,x)\in \cS'\}$ has dimension less than $m$ for each $t$.
Indeed, let $\cU_{\tau }$ be the cell
$$
\{(t,x)\in (0,1)^{k+m} \mid (t, x_{<m}) \in V_\tau \mbox{ and } (\psi_{\tau }(t, x_{<m})_{<m},x_m) \in \cU\}
$$
and let $\varphi_{\tau }:\cU_{\tau } \to \cX $ be the map 
$$
(t,x)\mapsto (\psi_{\tau }(t, x_{<m} )_{<m} , x_m , f ( \psi_{\tau }(t, x_{<m} )_{<m} , x_m ) ).
$$
By the above application of the induction hypothesis the function
$$
(t,x)\mapsto b( \psi_{\tau }(t, x_{<m} )_{<m} , x_m )
$$
on $\cU_{\tau }$ is a-b-m with an associated bounded-monomial map $d_{\tau}$ with bounded $C^1$-norm and with only integer exponents.
Let $b_{\tau}$ be the map $(c_{\tau},d_{\tau})$.
Then, the maps $\varphi_{\tau }$ satisfy \ref{pre1r}, \ref{pre2r}, \ref{pre3r} and \ref{pre4r} with associated bounded-monomial maps $b_{\tau }$. 
This finishes the proof of Theorem \ref{recti-pre-param}.
\end{proof}

We can now give the proof of Theorem \ref{thm:fibers-a-chart:g:g}.

\begin{proof}[Proof of Theorem \ref{thm:fibers-a-chart:g:g}]
The theorem follows almost directly from the rectilinear pre-parameterization result \ref{recti-pre-param}. Indeed, by Theorem \ref{recti-pre-param} we can reduce to the situation that $\cX\subset (0,1)^{k+m}$ is an open cell such that $X_t$ is $\ell$-rectilinear for each $t\in T$ and such that all walls of $\cX$ are a-b-m with an associated bounded-monomial map $b$ with only integer exponents such that moreover $b$ has bounded $C^1$-norm. This reduction involves working piecewise, and, a subanalytic Lipschitz-continuous transformation of $T$ which is harmless because of the Lipschitz continuity (recall how Lipschitz continuity is used in the proof of Theorem \ref{thm:semi-a-chart:g}). Now, let $S$ be ${\overline {T}}\setminus T$, where ${\overline {T}}$ is the topological closure of $T$ in $\RR^k$. Choose $\delta>0$. If $t$ lies in $T\setminus S_\delta$, then one has by the the rectilinear form and since the walls are a-b-m that
$$
c\delta^M< x_1, \ldots,c\delta^M< x_\ell
$$
for each $x\in X_t$ and some $c>0$ and $M\geq 1$ which are independent of $\delta$. Note also that the exponents of $x_{\ell+i}$ in $b$ must be nonnegative integers for any $i>0$ by the rectilinear form and the fact that $b$ has bounded range. Now, we are done by a variant of Lemma \ref{lem:b-m-a-chart} which takes the rectilinear form and the special nature (as integers, some known to be non-negative) of the exponents into account. 
\end{proof}

\subsection{}\label{sec:weier:sys}
 We end with a further refinement, using a more flexible notion of sets in $\RR^n$ of 
\cite{M06} than the one of power-subanalytic sets, which we now recall. 

Let $\cF$ be a Weierstrass system and let $\cL_\cF$ be the corresponding language as in \cite[Definition 2.1]{M06}. 
By the field of exponents of $\cF$ is meant the set of real $r$ such that $(0,1)\to\RR:x\mapsto (1+x)^r$ is $\cL_\cF$-definable; this set is moreover a field by Remark 2.3.5 of \cite{M06}.
Let $K$ be a subfield of the field of exponents of $\cF$. We denote by $\cL_\cF^K$ the expansion of $\cL_\cF$ by the functions
$$
x \mapsto \begin{cases} x^\rho, \mbox{ if } x> 0,\\ 0 \mbox{ otherwise, }
\end{cases}
$$
for each $\rho\in K$.

Now, we can refine the above theorem \ref{thm:semi-a-chart:g} 
as follows. If the initial data of $T$ and $\cX$ of Theorem \ref{thm:semi-a-chart:g} are moreover $\cL_\cF^K$-definable, then $S$ and the maps $\psi_i$ can be chosen to be $\cL_\cF^K$-definable as well. Theorem \ref{thm:semi-alg-a-chart} thus follows by using $K=\Q$ and the minimal choice of $\cF$, see \cite{M06}. A similar refinement of Theorem \ref{thm:fibers-a-chart:g:g} (giving $\cL_\cF$-definability of $S$ and the $\psi_i$) would also follow from the corresponding adaptation of Theorem 1.5 of \cite{CM13}, which we leave for the future. One may also expect that if the $X_t$ are topologically closed, the a-charts from Theorems \ref{thm:semi-alg-a-chart}, \ref{thm:fibers-semi-alg-a-chart}, \ref{thm:semi-a-chart:g} and \ref{thm:fibers-a-chart:g:g} can be taken with ranges contained in $X_t$.

\end{document}